\documentclass[12pt,a4paper]{article}
\usepackage {amsmath, amssymb, amsthm, graphicx, amsfonts, dsfont, enumerate, makeidx, verbatim, algorithm, algorithmic, mathtools, kbordermatrix, thmtools, mathdots, qbordermatrix, color}
\usepackage[utf8]{inputenc}
\usepackage[margin=1in]{geometry}

\newcommand\coolunder[2]{\mathrlap{\smash{\underbrace{\phantom{%
    \begin{matrix} #2 \end{matrix}}}_{\mbox{$#1$}}}}#2}

\newcommand\coolrightbrace[2]{%
\left.\vphantom{\begin{matrix} #1 \end{matrix}}\right\}#2}
\newtheorem{theorem}{Theorem}[section]
\newtheorem{lemma}[theorem]{Lemma}
\newtheorem{proposition}[theorem]{Proposition}
\newtheorem{corollary}[theorem]{Corollary}
\newtheorem{conjecture}[theorem]{Conjecture}

\theoremstyle{definition}
\newtheorem{definition}[theorem]{Definition}
\declaretheorem[style=definition,qed=$\diamondsuit$,sibling=definition]{example}
\renewcommand\thmcontinues[1]{Continued}
\newcommand{\nvar}{[x_1,\ldots,x_n]}

\newcommand{\mo}[1]{\ (\operatorname{mod} #1)}

\newcommand{\lc}[1]{\operatorname{lc} (#1)}

\newcommand{\ini}[1]{\operatorname{in}(#1)}

\numberwithin{algorithm}{section}
\renewcommand{\kbldelim}{(}
\renewcommand{\kbrdelim}{)}
\allowdisplaybreaks

\begin{document}
\title {Gröbner Bases of Generic Ideals}
\author {Juliane Capaverde and Shuhong Gao}
\maketitle

\begin{abstract}
Let $I = ( f_1, \dots, f_n )$ be a homogeneous ideal in the polynomial ring $K[x_1, \dots,x_n]$ over a field $K$ generated by generic polynomials. Using an incremental approach based on a method by Gao, Guan and Volny, and properties of the standard monomials of generic ideals, we show how a Gröbner basis for the ideal $(f_1, \dots, f_i)$ can be obtained from that of $(f_1, \dots, f_{i-1})$. If $\deg f_i = d_i$, we are able to give a complete description of the initial ideal of $I$ in the case where $d_i \geq \left(\sum_{j=1}^{i-1}d_j\right) - i -1$. It was conjectured by Moreno-Socías that the initial ideal of $I$ is almost reverse lexicographic, which implies a conjecture by Fröberg on Hilbert series of generic algebras. As a result, we obtain a partial answer to Moreno-Socías Conjecture: the initial ideal of $I$ is almost reverse lexicographic if the degrees of generators satisfy the condition above. This result improves a result by Cho and Park. We hope this approach can be strengthened to prove the conjecture in full.
\end{abstract}

\section{Introduction}

Let $R = K\nvar$ be the polynomial ring in $n$ variables over an infinite field $K$. A homogeneous ideal $I = (f_1, \dots, f_r)$ is said to be of type $(n, d_1, \dots, d_r)$ if, for each $1 \leq i \leq r$, $f_i$ is a homogeneous polynomial of degree $d_i$. We are interested in answering the following question: fixing the graded reverse lexicographic order, what is the shape of the Gröbner bases of \emph{most} ideals of type $(n, d_1, \dots, d_r)$?

An ideal $I = (f_1, \dots, f_r)$ of type $(n,d_1, \dots, d_r)$ can be identified, through the coefficients of $f_1, \dots, f_r$, with a point of the affine space $V = R^{N_1} \times R^{N_2} \times \cdots \times R^{N_r}$, where $N_i = {d_i+n-1 \choose d_i}$. That is, $V$ is viewed as the set of all ideals of type $(n, d_1, \dots, d_r)$. A property is said to be \emph{generic} if it holds on a Zariski-open subset of $V$. We also say such a property holds for \emph{most} ideals of type $(n, d_1, \dots, d_r)$, or that it holds for \emph{generic ideals}.

It is known that Hilbert functions and initial ideals are constant in a nonempty open subset of $V$ (see \cite{MR2651558}). A longstanding conjecture by Fröberg gives a formula for the generic Hilbert series.

\begin{conjecture}[Fröberg \cite{MR813632}] \label{conjfroberg} \index{Fröberg Conjecture}
If $I$ is a generic ideal of type $(n, d_1, \dots, d_r)$, then the Hilbert series of $R/I$, $S_{R/I}(z)$, is given by
$$S_{R/I}(z) = \bigg| \frac{\prod_{i=1}^r (1- z^{d_i})}{(1-z)^n}\bigg|.$$
\end{conjecture}

The notation above means the following: if $\sum_{d=0}^{\infty} a_dz^d$ is a power series with integer coefficients, then
$\big| \sum_{d=0}^{\infty} a_dz^d \big| = \sum_{d=0}^{\infty} b_dz^d$,
where $b_d = a_d$ if $a_i >0$  for $0 \leq i \leq d$, and $b_d=0$ otherwise.

Moreno-Socías gave a conjecture describing the initial ideal of generic ideals with respect to the graded reverse lexicographic order, related to the weak reverse lexicographic property.

\begin{definition} \index{almost reverse lexicographic}
Let $J = ( x^{\alpha_1}, \dots, x^{\alpha_r})$ be a monomial ideal, and suppose $x^{\alpha_1}, \dots, x^{\alpha_r}$ are minimal generators, that is, these monomials are not divisible by one another. $J$ is said to be \emph{almost reverse lexicographic}, or \emph{weakly reverse lexicographic}, if, for every $i$, $J$ contains every monomial $x^{\alpha}$ such that $\deg x^{\alpha} = \deg x^{\alpha_i}$ and $x^{\alpha} > x^{\alpha_i}$.
\end{definition}

\begin{conjecture}[Moreno-Socías \cite{MR1966660}]\label{moreno-socias_conjecture} \index{Moreno-Socías Conjecture}
If $I$ is a generic homogeneous ideal in $R$, then the initial ideal of $I$ is almost reverse lexicographic.
\end{conjecture}

The following is a weaker version of the Moreno-Socías conjecture, restricted to generic ideals in $R = K[x_1, \dots, x_n]$ generated by $n$ generic forms.

\begin{conjecture}\label{msconj}
Let $I = ( f_1, \dots, f_n )$ be a generic ideal in $R$. Then $\ini{I}$ is almost reverse lexicographic.
\end{conjecture}

It turns out that Conjecture \ref{msconj} implies the case where the number $r$ of polynomials is different from the number of variables (see \cite{MR2651558} for a proof of this fact). In this paper, only the case where $r=n$ is treated.

Partial answers have been given to both conjectures. Moreno-Socías Conjecture was proven to hold in the case $n=2$ by Aguirre \emph{et al.} \cite{MR2044588} and Moreno-Socías \cite{MR1966660}, $n=3$ by Cimpoea{\c{s}} \cite{MR2292574}, $n=4$ by Harima and Wachi \cite{MR2573232}, and the case for $d_1, \dots, d_n$ satisfying $d_i > \sum_{j=1}^i d_j - i +1$ by Cho and Park \cite{MR2397406}. The Fröberg Conjecture is known to be true for $n=2$ \cite{MR813632}, $n=3$ \cite{MR839581}, $r = n+1$ \cite{MR813632}, and $d_1 = \cdots = d_r =2$ and $n\leq 11$, $d_1 = \cdots = d_r = 3$ and $n \leq 8$ \cite{MR1283740}.

It was pointed out in \cite{MR2397406} that the Moreno-Socías Conjecture and the Fröberg Conjecture are closely related to the Lefschetz and Stanley properties.

\begin{definition}
Let $I$ be a zero-dimensional homogeneous ideal in $R$, and $A = R/I$. Let $\delta = \max\{i\ |\ A_i \neq 0\}$.
\begin{enumerate}[(i)]
\item $A$ has the strong (resp. weak) Lefschetz property \index{Lefschetz property} if there is a linear form $L$ such that multiplication map
$$A_{i} \overset{L^d}{\longrightarrow} A_{i+d}$$
is of maximal rank (that is, it is injective if $\dim_K A_i \leq \dim_K A_{i+d}$, and surjective if $\dim_K A_i \geq \dim_K A_{i+d}$) for each $i$ and each  $d$ (resp. d=1).

\item $A$ has the strong Stanley property \index{Stanley property} if there is a linear form $L$ such that the multiplication map
$$A_{i} \overset{L^{\delta-2i}}{\longrightarrow} A_{\delta-i}$$
is bijective for each $0 \leq i \leq \lfloor \delta/2 \rfloor$.
\end{enumerate}
\end{definition}

In \cite{MR2651558} Pardue shows that the Moreno-Socías conjecture implies a series of other conjectures. In particular, it implies the Fröberg conjecture, a fact that was also proven by Cho and Park \cite{MR2397406}. The connection between Moreno-Socías Conjecture and Fröberg Conjecture and the Lefschetz property is illustrated by the following conjecture, which is equivalent to the Fröberg Conjecture.

\begin{conjecture}[Pardue \cite{MR2651558}]\index{Pardue Conjecture}
If $I$ is a generic ideal of type $(n, d_1, \dots, d_n)$, then $x_{n-1}$ is a weak Lefschetz element on $R/\ini{I}+(x_n, \dots, x_{n-i+1})$, for each $0 \leq i \leq n-1$.
\end{conjecture}

In \cite{MR2044588}, Aguirre et al gave the Gröbner basis of generic ideals in the case $n=2$. They used Buchberger's Criterion to prove a given set, which was suggested by computer calculations, was indeed a Gröbner bases. Our approach is similar to theirs, but instead of using Buchberger's Algorithm, we apply an incremental method from \cite{Gao10anew}. In this paper, we show how a Gröbner basis for the generic ideal $( I,g )$ can be obtained from the Gröbner basis of $I$ when a generic polynomial $g$ is added, employing properties of the standard basis of $I$. More precisely, we give a proof that the quotient $R/I$ has the strong Lefschetz and strong Stanley properties, and use this fact to establish the form of the initial ideal of $( I,g )$. As a result, we obtain a description of the initial ideal of $I = ( f_1, \dots, f_n )$ in the case the degrees $d_1, \dots, d_n$ satisfy $d_i \geq \left(\sum_{j=1}^{i-1} d_j\right) -i -1$. Our construction shows that Moreno-Socías Conjecture is true for these ideals, thus we give a partial answer to the conjecture. Our result is somewhat more general then the one given by Cho and Park in \cite{MR2397406}, where they showed Moreno-Socías to be true for degrees satisfying $d_i > \left(\sum_{j=1}^{i-1} d_j\right) -i +1$. We expect that our method can be strengthened to fully prove the conjecture.

\section{Incremental Gröbner bases}\label{g2v}

Let $I$ be any ideal in $R$ and suppose $G$ is a Gröbner basis for $I$ with respect to some monomial order. The initial term of a polynomial $f$ will be denoted by $\ini{f}$, and the initial ideal of $I$ will be denoted by $\ini{I}$. Let $g$ be any polynomial in $R$. We now describe the method given in \cite{GGV10} to obtain a Gröbner basis of the ideal $( I, g)$.

Let $B = \{ x^{\alpha_1}, x^{\alpha_2}, \dots, x^{\alpha_N}\}$ be the set of monomials that are not in $\ini{I}$, which we call standard monomials with respect to $I$. Note that when $I$ is not zero-dimensional, we have $N= \infty$.

Suppose $x^{\alpha_i}g \equiv h_i \mo G$, where $h_i \in R$ is a $K$-linear combination of monomials in $B$, for $1 \leq i \leq N$. We can write this as
\begin{equation}\label{GGV3}
\begin{pmatrix}
  x^{\alpha_1} \\
  x^{\alpha_2} \\
  \vdots \\
  x^{\alpha_N} \\
\end{pmatrix}
\cdot g \equiv \begin{pmatrix}
  h_1 \\
  h_2 \\
  \vdots \\
  h_N \\
\end{pmatrix}
\mo G.
\end{equation}

We apply row operations to both sides of Equation (\ref{GGV3}) as follows: for $1 \leq i < j \leq N$ and $a \in K$, subtract from the $j$-th row the $i$-th row multiplied by $a$. Our goal is to eliminate equal leading terms. So if $\ini{h_i} = \ini{h_j}$, with $i<j$, we use a row operation to eliminate the leading term of $h_j$. This means we only perform row operations downward. We start with $h_1$, using the first row to eliminate the leading term of all $h_j$ bellow that have the same leading monomial as $h_1$. Then we pass to the leading monomial of the new second row, and eliminate the leading terms of all $h_j$'s with the same leading monomial. Then we go to the new third row, and so on. Since the monomial order is a well ordering, any decreasing sequence of monomials must be finite. Hence we perform only a finite number of row operations on row $j$, using rows above it. By induction, we may assume that Equation (\ref{GGV3}) can be transformed into the form
\begin{equation}\label{GGV5}
\begin{pmatrix}
  u_1 \\
  u_2 \\
  \vdots \\
  u_N \\
\end{pmatrix}
\cdot g \equiv \begin{pmatrix}
  v_1 \\
  v_2 \\
  \vdots \\
  v_N \\
\end{pmatrix}
\mo G
\end{equation}
where $u_i, v_i \in R$ are $K$-linear combinations of monomials in $B$, and for $1 \leq i < j \leq N$ with $v_i, v_j \neq 0$, we have $\ini{v_i} \neq \ini{v_j}$, that is, the nonzero rows in the right-hand side of (\ref{GGV5}) have distinct leading monomials.

\begin{theorem}[Gao, Guan and Volny \cite{GGV10}]
Let $\widetilde{G} = G \cup \{ v_i | 1\leq i \leq N \}$. Then $\widetilde{G}$ is a Gröbner basis of $( I, g )$.
\end{theorem}

\section{Gröbner bases of generic ideals}\label{genericsb}

Now, let us return to the generic setting. We use here the definition of generic ideals from \cite{MR1283740}, which is more suitable for our approach. Suppose $K$ is an extension of a base field $F$.

\begin{definition}\index{generic polynomial} \begin{enumerate}[{\normalfont (i)}]
\item A polynomial $f \in R$ of degree $d$ is called \emph{generic} over $F$ if
$$f = \sum_{\alpha} c_{\alpha} x^{\alpha},$$
where the sum runs over all monomials of degree $d$ in $R$, and the coefficients $c_{\alpha}$ are algebraically independent over $F$.

\item An ideal $I \subset R$ \index{generic ideal}is generic if it is generated by generic polynomials $f_1, \dots, f_r$ with all the coefficients algebraically independent over $F$. $I$ is a generic ideal of type $(n, d_1, \dots, d_r)$ if $I$ is a generic ideal in $R$ generated by generic polynomials of degree $d_1, \dots, d_r$.
\end{enumerate}
\end{definition}

We apply the method above to generic ideals. Let $S = K[x_1, \dots, x_n, z]$ denote the polynomial ring in $n+1$ variables. Let $f_1, \dots, f_n$ and $g$ be generic polynomials in $S$ such that $I = (f_1, \dots, f_n)$ is a generic ideal of type $(n+1, d_1, \dots, d_n)$, and $(f_1, \dots, f_n, g)$ is a generic ideal of type $(n+1, d_1, \dots, d_n, d)$. Let $G$ denote the reduced Gröbner basis of $I$ with respect to the grevlex order. Let $E$ be the set of standard monomials with respect to $I$. We denote by $\mathbf{E}$ the column vector whose entries are the monomials in $E$ listed in decreasing order. The set of elements of degree $i$ in $E$ is denoted by $E_i$.

Reducing $g$ modulo $G$ we obtain a polynomial that is a linear combination of monomials of degree $d$ in $E$ with coefficients that are still algebraically independent over $F$, and also algebraically independent over the extension of $F$ generated by the coefficients of elements of $G$. Thus, from now on we assume that $g$ is reduced modulo $G$, that is, we take $g$ to be a linear combination of monomials in $E_d$ with coefficients algebraically independent over the extension of $F$ generated by coefficients of elements of $G$.

Let $M$ be the matrix satisfying
\begin{equation} \label{matrixform} \mathbf{E} g \equiv M \mathbf{E} \mo G. \end{equation}

Note that all polynomials involved are homogeneous. So, for a monomial $m \in E_i$, the product $mg$ is homogeneous and its reduced form is a homogenous polynomial of degree $i+d$, that is, a $K$-linear combination of monomials in $E_{i+d}$ only. Also, the row operations can only be performed using two rows containing polynomials of the same degree. Thus, we consider rows of different degrees separately. Let $M_i$ denote the matrix such that
\begin{equation}\label{matrixmi} \mathbf{E}_i g \equiv M_i \mathbf{E}_{i+d} \mo G, \end{equation}
where $\mathbf{E}_i$ denotes the column vector whose entries are the monomials in $E_i$ listed in decreasing order.

\begin{lemma}
The rows of $M_i$ are linearly independent.
\end{lemma}
\begin{proof}
Denote the rows of $M_i$ by $\mathbf{v}_1, \dots, \mathbf{v}_{\ell}$, and suppose $\mathbf{E}_i = (m_1, \dots, m_{\ell})^{\mathrm{T}}$. Assume $c_1\mathbf{v}_1 + \cdots + c_{\ell}\mathbf{v}_{\ell} =0$, with $c_1, \dots, c_{\ell} \in K$. Then
$$ (c_1 m_1 + \cdots + c_{\ell} m_{\ell})g \equiv c_1\mathbf{v}_1 + \cdots + c_{\ell}\mathbf{v}_{\ell} \equiv 0 \mo G.$$
Since $g$ is regular, $g$ is not a zero divisor in $R/I$, and it follows that $c_1m_1 +\cdots + c_{\ell}m_{\ell} = 0$ in $R/I$. Since the monomials in $E_i$ are $K$-linear independent, it follows that $c_j =0$ for all $1 \leq j \leq \ell$. Hence, $\mathbf{v}_1, \dots, \mathbf{v}_{\ell}$ are linearly independent.
\end{proof}

Thus, each matrix $M_i$ has rank $|E_i|$. To be able to describe $\ini{I,g}$, we need to see which columns are linearly independent. The monomials in $E_{i+d}$ corresponding to the linearly independent columns are the ones that will be added to $\ini{I}$ to form $\ini{I,g}$. Note that some of the monomials might be redundant, that is, multiples of monomials previously added to the basis. For the Moreno-Socías Conjecture to be true, the columns corresponding to the greatest non-redundant monomials in $E_{i+d}$ should be linearly independent.

Define $\pi : S \longrightarrow R$ to be the ring homomorphism that takes $z$ to zero, fixing the elements in $K$ and the variables $x_1, \dots, x_n$. Let $J = \pi(I) \subset R$ be the image of $I$. Then $J$ is a generic ideal of type $(n, d_1, \dots, d_n)$ in $R$, and using a property of the reverse lexicographic order we have that $\pi(\ini{I}) = \ini{J}$ (see \cite[Proposition 15.12]{MR1322960}. Moreover, since $z$ is regular in $S/I$, by Theorem 15.13 from \cite{MR1322960}, $z$ is regular in $S/\ini{I}$, and by Proposition 15.14, also from \cite{MR1322960}, we have that $\ini{I}$ is generated by monomials that are not divisible by $z$. Thus, $\ini{I}$ and $\ini{J}$ have the same minimal generators. Let $B \subset R$ denote the set of standard monomials with respect to $J$. It follows that
\begin{eqnarray*}
E & = & \{mz^{\ell} | m \in B,\ \ell\geq 0 \} \\
  & = & B \cup zB \cup z^2B \cup z^3B \cup \cdots .
\end{eqnarray*}

Now, $J$ is a generic ideal in $R$ generated by $n$ generic forms, and for this type of ideals some properties are known. For instance, $B_i = 0$, for all $i > \delta$, where $\delta = d_1 + \cdots + d_n - n$. So, for each $0 \leq i \leq \delta$,
$$E_i = B_i \cup zB_{i-1} \cup z^2B_{i-1} \cup \cdots \cup z^{i-1}B_1 \cup z^iB_0,$$
and for $i \geq \delta$,
$$E_i = z^{i-\delta}B_{\delta} \cup z^{i-\delta+1}B_{\delta-1} \cup \cdots \cup z^{i-1}B_1 \cup z^iB_0.$$
This implies that for $i > \delta$, $\mathbf{E}_i = z^{i-\delta} \mathbf{E}_{\delta}$ and
$$ \mathbf{E}_i g \equiv z^{i-\delta} M_{\delta} \mathbf{E}_{\delta+d} \mo G.$$
Thus, the Gröbner basis elements obtained at this point are redundant, and we only need to consider $\mathbf{E}_i g$ for $0 \leq i \leq \delta$.

Also note that, since $E_{i} = zE_{i-1} \cup B_{i}$, $M_{i-1}$ is a submatrix of $M_{i}$ for all $1 \leq i \leq \delta$, and all matrices $M_i$ are submatrices of $M_{\delta}$. The rows and columns of $M_{\delta}$ can be indexed by the elements of $B$. For $0 \leq i \leq \delta$, $M_i$ is formed by blocks $\Gamma_{j,k}$, for $0 \leq j \leq i$ and $0 \leq k \leq d+i$, where the entries of $\Gamma_{j,k}$ are the coefficients of the monomials in $z^{d+i-k}B_k$ in the reduced form of the polynomials in $z^{i-j}gB_j$. So $M_i$ can be written as
\[ M_{i} =
  \qbordermatrix[()]{ne}{%
     \mathbf{B}_{d+i} & \mathbf{B}_{d+i-1} & \cdots & \mathbf{B}_{1} & \mathbf{B}_{0} & \cr
     \Gamma_{i,d+i} & \Gamma_{i,d+i-1}& \cdots & \Gamma_{i,1} & \Gamma_{i,0} & \mathbf{B}_{i} \cr
      \Gamma_{i-1,d+i} & \Gamma_{i-1,d+i-1} & \cdots & \Gamma_{i-1,1} & \Gamma_{i-1,0} & \mathbf{B}_{i-1} \cr
      \vdots & \vdots & \ddots & \vdots & \vdots & \vdots \cr
      \Gamma_{1,d+i} & \Gamma_{1,d+i-1} & \cdots & \Gamma_{1,1} & \Gamma_{1,0} & \mathbf{B}_1 \cr
      \Gamma_{0,d+i} & \Gamma_{0,d+i-1} & \cdots & \Gamma_{0,1} & \Gamma_{0,0} & \mathbf{B}_{0} \cr
  },
\]
and Equation (\ref{matrixmi}) takes the form
\begin{equation*}
\begin{pmatrix}
  \mathbf{B}_i \\
  z\mathbf{B}_{i-1} \\
  \vdots \\
  z^i\mathbf{B}_{0} \\
\end{pmatrix} g \equiv \begin{pmatrix}
                         \Gamma_{i,d+i} & \Gamma_{i,d+i-1} & \cdots & \Gamma_{i,0} \\
                         \Gamma_{i-1,d+i} & \Gamma_{i-1,d+i-1} & \cdots & \Gamma_{i-1,0} \\
                           &   & \ddots  &   \\
                         \Gamma_{0,d+i} & \Gamma_{0,d+i-1} & \cdots & \Gamma_{0,0} \\
                       \end{pmatrix} \begin{pmatrix}
                                       \mathbf{B}_{d+i} \\
                                       z\mathbf{B}_{d+i-1} \\
                                       \vdots \\
                                       z^{d+i}\mathbf{B}_0 \\
                                     \end{pmatrix} \mo G.
\end{equation*}

\subsection{Standard bases of generic ideals and the structure of the blocks $\Gamma_{j,k}$}

Our goal is to study the structure of the blocks $\Gamma_{j,k}$. First we establish some properties of the sets $B_i$.

As noted before, the ideal $J = \pi(I)$ is a generic ideal in $R$, generated by the generic polynomials $\pi(f_1), \dots, \pi(f_n)$, with $\deg(\pi(f_i)) = d_i$, for $1 \leq i \leq n$. Let $A = R/J$, and define
\begin{eqnarray*}
\delta & = & d_1 + \cdots + d_n - n, \\
\delta^* & = & d_1 + \cdots + d_{n-1} - (n-1), \\
\sigma & = & \min\{\delta^*, \lfloor \delta/2\rfloor\}, \\
\mu & = & \delta - 2 \sigma.
\end{eqnarray*}
The Hilbert series of $A$ is known to be a symmetrical polynomial of degree $\delta$, given by
$$ S(z) = \frac{\prod_{j=1}^{n} (1-z^{d_j})}{(1-z)^n} = \sum_{\nu=0}^{\delta} a_{\nu}z^{\nu},$$
with $0 < a_0 < \cdots < a_{\sigma} = \cdots = a_{\sigma + \mu} > \cdots > a_{\delta} > 0$ (see \cite[Proposition 2.2]{MR1966660}). Here $a_{\nu} = |B_{\nu}|$.

To prove the properties of $B$ we need in our proofs, we use a result from \cite{MR1966660}. For $e \geq 0$, define
\begin{equation*}
\widetilde{B}^e  =  \{ x_1^{\alpha_1}\cdots x_{n-1}^{\alpha_{n-1}} | x_1^{\alpha_1}\cdots x_{n-1}^{\alpha_{n-1}}x_n^e \in B \} \subset K[x_1, \dots, x_{n-1}].
\end{equation*}

\begin{proposition}[Moreno-Socías \cite{MR1966660}]\label{propMS} With the notation above,
$$ \widetilde{B}^0 = \widetilde{B}^1 = \cdots = \widetilde{B}^{\mu},$$
$$\widetilde{B}^{\mu+1} = \widetilde{B}^{\mu+2}, \dots, \widetilde{B}^{\delta-1} = \widetilde{B}^{\delta},$$
and
$$\widetilde{B}^{\delta-2\lambda} = \{ m \in \widetilde{B}^0 | \deg(m) \leq \lambda\},$$
for $0 \leq \lambda < \sigma$.
\end{proposition}



The following lemma implies that $A$ has the strong Stanley property, and $x_n$ is a strong Stanley element for $A$.

\begin{lemma}\label{propStanley} Let $0 \leq i \leq \frac{\delta}{2}$. Then $B_{\delta-i} = x_n^{\delta -2i} B_i$.
\end{lemma}
\begin{proof}
Note that $B_{\delta-i} = x_n^{\delta-2i} B_i$ if and only if $\widetilde{B}_{i}^e = \widetilde{B}_{\delta-i}^{e+\delta-2i}$, for all $e \geq 0$. For $e > i$, we have $ \widetilde{B}_i^e  = \widetilde{B}_{\delta-i}^{e+\delta-2i} = \emptyset$. So we need to show equality for $0 \leq e \leq i$.

If both $e \leq \mu$ and $e+\delta-2i \leq \mu$, then $\widetilde{B}_i^e = \widetilde{B}_{\delta-i}^{e+\delta-2i} = \widetilde{B}_{i-e}^0$.

If $e \leq \mu$ and $e + \delta-2i > \mu$, then $\widetilde{B}_i^e = \widetilde{B}_{i-e}^0$, and $\widetilde{B}_{\delta-i}^{e+\delta-2i} = \{ m \in \widetilde{B}_{i-e}^0 | \deg(m) \leq \lambda\}$, where $e+\delta-2i = \delta-2\lambda$ or $e+\delta-2i = \delta-2\lambda-1$. We need to see that $i-e \leq \lambda$. In the first case we have
$$\lambda = \frac{2i-e}{2} \geq \frac{2i-2e}{2}= i-e,$$
and in the second case,
$$ \lambda = \frac{2i-e-1}{2} \geq \frac{2i-e-e}{2} = i-e,$$
as $e \geq 1$.

Now, if $e > \mu$, then $$\widetilde{B}^{e} = \{ m \in \widetilde{B}^0 | \deg(m) \leq \lambda\},$$ where $e = \delta-2\lambda$ or $e = \delta-2\lambda-1$, and $$\widetilde{B}^{e+\delta-2i} =\{ m \in \widetilde{B}^0 | \deg(m) \leq \lambda'\},$$ where $e +\delta-2i = \delta-2\lambda'$ or $e +\delta-2i = \delta-2\lambda'-1$. We want to see that $i-e \leq \lambda$ if and only if $i-e \leq \lambda'$.

Case 1: Suppose $e = \delta -2\lambda$ and $e+\delta-2i = \delta -2\lambda'$. This happens when $\delta$ is even, giving $\lambda'= \lambda+i-\frac{\delta}{2}$. If $ i-e \leq \lambda$, then $ \lambda' = \lambda + i -\frac{\delta}{2} = i - \frac{e}{2} \geq i-e$. If $ i-e \leq \lambda'$, then $\lambda \geq \lambda +i -\frac{\delta}{2} = \lambda' \geq i-e$, as $i-\frac{\delta}{2} \leq 0$.

Case 2: Suppose $e = \delta-2\lambda$ and $e+\delta-2i = \delta-2\lambda'-1$, with $\delta$ odd and $\lambda'=\lambda +i-\frac{\delta-1}{2}$. If $i-e \leq \lambda$, then $\lambda' = \lambda + i -\frac{\delta-1}{2} = i - \frac{e}{2} \geq i-e$. And if $ i-e \leq \lambda'$, then $ \lambda \geq \lambda +i -\frac{\delta-1}{2} = \lambda' \geq i-e$, as $i-\frac{\delta-1}{2} \leq 0$.

Case 3: Suppose $e = \delta-2\lambda-1$ and $e +\delta-2i = \delta-2\lambda'$, with $\delta$ odd and $\lambda'= \lambda+i-\frac{\delta+1}{2}$. If
$ i-e \leq \lambda$, then $ \lambda'= \lambda +i-\frac{\delta+1}{2} = i - \left( \frac{e}{2}+1 \right) \geq i-e$, as $e > \mu \geq 1$ ($\mu=0$ would contradict the fact that $\delta$ is odd). If $ i-e \leq \lambda'$, then $ \lambda = \lambda +i - \frac{\delta+1}{2} = i - \left( \frac{e}{2} +1 \right) \geq i-e$.

Case 4: Finally, Suppose $ e = \delta-2\lambda-1$ and $e+\delta-2i = \delta-2\lambda'-1$, with $\delta$ an even integer and $\lambda' = \lambda +i - \frac{\delta}{2}$. If $ i-e \leq \lambda$, then $ \lambda'= \lambda +i - \frac{\delta}{2} = i - \frac{e+1}{2} \geq i-e$, as $e > \mu \geq 0$. If
$ i-e \leq \lambda'$, then $\lambda \geq \lambda +i - \frac{\delta}{2} = i- \frac{e+1}{2} \geq i-e$.
\end{proof}

The next lemma establishes that $A$ has the strong Lefschetz property, and $x_n$ is a strong Lefschetz element for $A$.

\begin{lemma}\label{lefschetzprop}
Let $0 \leq j \leq \delta$ and $r \geq 0$. Then multiplication by $x_n^r$ from $A_j$ to $A_{j+r}$ is either injective or surjective. More precisely:
\begin{enumerate}[{\normalfont (i)}]
\item Suppose $|B_j| \leq |B_{j+r}|$. Let $S$ denote the subset of $B_{j+r}$ consisting of $|B_j|$ smallest monomials in $B_{j+r}$. Then
$$ S = x_n^r B_{j}. $$
\item Suppose $|B_j| \geq |B_{j+r}|$. Let $S$ denote the subset of $B_j$ consisting of $|B_{j+r}|$ smallest monomials in $B_j$. Then
$$ B_{j+r} = x_n^r S. $$
\end{enumerate}
\end{lemma}

\begin{proof}
First, suppose $0 \leq j \leq \delta/2$ and $j+r \leq \delta-j$. Then $|B_j| \leq |B_{j+r}|$. By Lemma \ref{propStanley}, $B_{\delta-j} = x_n^{\delta-2j}B_j$, so multiplication by $x_n^{\delta-2j}$ from $A_{j}$ to $A_{\delta-j}$ is bijective. This multiplication can be seen as the composition
$$ A_j \overset{x_n^r}{\longrightarrow} A_{j+r} \overset{x_n^{\delta-2j-r}}{\longrightarrow} A_{\delta-j}$$
so that multiplication by $x_n^r$ from $A_j$ to $A_{j+r}$ must be injective. Moreover, if $m$ is a monomial in $B_j$, $x_n^{\delta-2j}m$ is in $B_{\delta-j}$, which implies $x_n^r m \in B_{j+r}$. So, $x_n^r B_j \subseteq B_{j+r}$. Suppose $B_j = \{ x^{\alpha_1}, \dots, x^{\alpha_N}\}$, with $x^{\alpha_1} < \dots < x^{\alpha_N}$, and suppose $m$ is a monomial in $B_{j+r}$ such that $m < x^{\alpha_i}x_n^r$. Then $x_n^r$ divides $m$, and $m'= m / x_n^r \in B_j$, with $m'< x^{\alpha_i}$. This proves (i).

Now suppose $0 \leq j \leq \delta/2$ and $j+r \geq \delta-i$. Then $|B_j| \geq |B_{j+r}|$. Let $m$ be a monomial in $B_{j+r}$. Since $B_{j+r} = x_n^{2(j+r)-\delta} B_{\delta-j-r}$, we can write $m = x_n^{2(j+r)-\delta} m'$, for some monomial $m' \in B_{\delta-j-r}$. By the previous paragraph, $m'' = x_n^{2j+r-\delta} m' \in B_j$, so $m = x_n^r m''$. So multiplication by $x_n^r$ is surjective. Moreover, the monomials $m'' \in B_j$ that are taken to $m \in B_{j+r}$ are in the image of $B_{\delta-j-r}$ under multiplication by $x_n^{2j+r-\delta}$, and, by part (i), correspond to the smallest monomials in $B_j$.

If $\delta/2 \geq j \geq \delta$, then $|B_j| \geq |B_{j+r}|$, and the same argument from the previous paragraph works.
\end{proof}

Since $g$ is a combination of all monomials in $E_d = B_d \cup zB_{d-1} \cup \cdots \cup z^d B_0$, we can write
$$ g = \mathbf{v}_{d}\cdot \mathbf{B}_{d} + \mathbf{v}_{d-1}\cdot \mathbf{B}_{d-1}z + \cdots + \mathbf{v}_1 \cdot \mathbf{B}_1z^{d-1} + \mathbf{v}_0\cdot \mathbf{B}_0 z^d,$$
where $\mathbf{v}_i$ is a row vector of coefficients. We denote the last entry of $\mathbf{v}_i$ by $c_i$. The following Lemma gives some of the structure of the blocks $\Gamma_{j,k}$.

\begin{lemma}\label{blocklemma}
Let $0 \leq i \leq \delta$, $0 \leq j \leq i$ and $j \leq k \leq \delta$.
\begin{enumerate}[{\normalfont (i)}]
\item Suppose $|B_j| \leq |B_k|$. Then the entries on the diagonal of the square submatrix of $\Gamma_{j,k}$ formed by the last $|B_j|$ columns have the form $c_{k-j} + L$, where $L$ is linear on other coefficients in $\mathbf{v}_{k-j}$, $\mathbf{v}_{k-j+1}$, $\dots$, $\mathbf{v}_{\delta}$ and does not involve $c_{k-j}$. Also, $c_{k-j}$ does not appear in the other entries of $\Gamma_{j,k}$.
\begin{equation}\label{gammaform1}
\Gamma_{j,k} = \begin{pmatrix}
                 * & \cdots & * & c_{k-j} + L & * & \cdots & * \\
                 * & \cdots & * & * & c_{k-j} + L & \cdots & * \\
                  &  &  &  &  & \ddots &  \\
                 * & \cdots & * & * & * & \dots & c_{k-j} + L \\
               \end{pmatrix}
\end{equation}

\item Suppose $|B_j| \geq |B_k|$. Then the entries on the diagonal of the square submatrix of $\Gamma_{j,k}$ formed by the last (bottom) $|B_k|$ rows have the form $c_{k-j} + L$, where $L$ is linear on other coefficients in $\mathbf{v}_{k-j}, \mathbf{v}_{k-j+1}, \dots, \mathbf{v}_{\delta}$ and does not involve $c_{k-j}$.
\begin{equation}\label{gammaform2}
  \Gamma_{j,k} = \begin{pmatrix}
                   * & * & \cdots & * \\
                    &  & \vdots &  \\
                   * & * & \cdots & * \\
                   c_{k-j} + L & * & \cdots & * \\
                   * & c_{k-j} + L & \cdots & * \\
                    &  & \ddots &   \\
                   * & * & \cdots & c_{k-j} + L \\
                 \end{pmatrix}
\end{equation}

\end{enumerate}
\end{lemma}

\begin{proof}
(i) Let $x^{\alpha} \in B_j$, and consider the term $c_{k-j}x_n^{k-j}z^{d+j-k}$ of $g$. By Lemma \ref{lefschetzprop}, the monomial $x^{\alpha}x_n^{k-j}$ is in $B_k$, that is, it is reduced modulo $G$. So in the reduced form of the product $x^{\alpha}z^{i-j} \cdot g$, $c_{k-j}$ will certainly appear in the coefficient of the monomial $x^{\alpha}x_n^{k-j}z^{d+i-k}$. Larger monomials that appear in the product might not be reduced, and the reduction would result in a coefficient of the form $c_{k-j} + L$, as claimed. Since the coefficient $c_{k-j}$ comes from a unique term in $g$, it cannot appear in any other entries.

(ii) Again, we let $x^{\alpha}$ be a monomial in $B_j$. Suppose that $x^{\alpha}$ is among the $|B_k|$ smallest monomials in $B_j$. By Lemma \ref{lefschetzprop}, the monomial $x^{\alpha}x_n^{k-j}$ is in $B_k$, so $c_{k-j}$ appears in the coefficient of the monomial $x^{\alpha}x_n^{k-j}z^{d+i-k}$ in the reduced form of $x^{\alpha}z^{i-j}\cdot g$. In the reduction process, possibly larger terms will be reduced resulting in a coefficient of the form $c_{k-j} + L$. Note that $c_{k-j}$ might appear in the top rows of $\Gamma_{k,j}$, that is, $c_{k-j}$ appears only once in each of $|B_k|$ the bottom rows, but we cannot guarantee it does not appear in other entries in the top rows.
\end{proof}

\subsection{Main results}

In the next lemmas we handle the case with $\delta-d \leq i \leq \delta$. Let $\Theta_i$ denote the square submatrix of $M_i$ formed by columns corresponding to the $|E_i|$ largest monomials in $E_{i+d}$. We want to show that $\Theta_i$ is nonsingular, for all $0 \leq i \leq \delta$. The determinant of $\Theta_i$ is a polynomial in the coefficients of $g$. We need to see that this polynomial is nonzero. Our goal is to show there is a term that can be obtained as a product of entries in a unique way, and hence cannot be cancelled. In this case the $|E_i|$ largest monomials in $E_{d+i}$ are the monomials in $z^{i+d-\delta}B_{\delta}, z^{i+d-\delta+1}B_{\delta-1}, \dots, z^{2i+d-\delta}B_{\delta-i}$, and $\Theta_i$ is formed by the following blocks
\[ \Theta_{i} =
  \qbordermatrix[()]{ne}{%
     \mathbf{B}_{\delta} & \mathbf{B}_{\delta-1} & \cdots & \mathbf{B}_{\delta-i+1} & \mathbf{B}_{\delta-i} & \cr
     \Gamma_{i,\delta} & \Gamma_{i,\delta-1}& \cdots & \Gamma_{i,\delta-i+1} & \Gamma_{i,\delta-i} & \mathbf{B}_{i} \cr
      \Gamma_{i-1,\delta} & \Gamma_{i-1,\delta-1} & \cdots & \Gamma_{i-1,\delta-i+1} & \Gamma_{i-1,\delta-i} & \mathbf{B}_{i-1} \cr
      \vdots & \vdots & \ddots & \vdots & \vdots & \vdots \cr
      \Gamma_{1,\delta} & \Gamma_{1,\delta-1} & \cdots & \Gamma_{1,\delta-i+1} & \Gamma_{1,\delta-i} & \mathbf{B}_1 \cr
      \Gamma_{0,\delta} & \Gamma_{0,\delta-1} & \cdots & \Gamma_{0,\delta-i+1} & \Gamma_{0,\delta-i} & \mathbf{B}_{0} \cr
  }.
\]

\begin{lemma}\label{lemmaselec}
Suppose $\delta - d \leq i \leq \delta$, and $i \geq \delta/2$. Then the term
\begin{equation}\label{term1} c_{\delta-i}^{(i+1)a_0} c_{\delta-i-1}^{i(a_1-a_0)} c_{\delta-i-2}^{(i-1)(a_2-a_1)} \cdots c_{0}^{(2i-\delta+1)(a_{\delta-i}-a_{\delta-i-1})} \end{equation}
can be obtained from the product of entries of $\Theta_i$, with exactly one entry from each column and row.
\end{lemma}

\begin{proof}

We will show how to select entries from $\Theta_i$ in steps. In each step, we pick entries from a certain set of blocks $\Gamma_{j,k}$. We start at step 0, selecting entries from the blocks on the diagonal of $\Theta_i$, and then blocks above the diagonal in the next step, and so on.

Let $0 \leq \ell \leq \delta-i$. At step $\ell$ we select $a_{\ell} - a_{\ell-1}$ entries from blocks
\begin{equation} \label{blocksdiag} \Gamma_{i,\delta-\ell}, \Gamma_{i-1, \delta-\ell+1}, \dots, \Gamma_{\ell, \delta-i}. \end{equation}
The entries selected are the ones in the bottom $a_{\ell}$ rows, skipping the bottom $a_{\ell-1}$, and $a_{\ell}$ right-most columns, skipping the last $a_{\ell-1}$ columns. These entries have the form $c_{\delta-\ell-i} + L$.
\[ \vphantom{
    \begin{matrix}
    \overbrace{XYZ}^{\mbox{$R$}}\\ \\ \\ \\ \\ \\
    \underbrace{pqr}_{\mbox{$S$}}
    \end{matrix}}%
\begin{pmatrix}
\ddots &  &  &  &  &  & \\
 & c_{\delta-\ell-i +L} & * & * & * & * & * \\
 & * & \underline{c_{\delta-\ell-i +L}} & * & * & * & * \\
 & * & * & \underline{c_{\delta-\ell-i +L}} & * & * & * \\
 & * & * & * & \underline{c_{\delta-\ell-i +L}} & * & * \\
 & * & * & * & * & c_{\delta-\ell-i +L} & * \\
 & * & \coolunder{a_{\ell}-a_{\ell-1}}{ \qquad * \qquad & \qquad * \qquad & \qquad * \qquad} & \coolunder{a_{\ell -1}}{ \qquad * \qquad & c_{\delta-\ell-i +L}}
\end{pmatrix}%
\begin{matrix*}[l]
\vphantom{\ddots}\\ \vphantom{a}\\
\coolrightbrace{x \\ x \\ y}{a_{\ell}-a_{\ell-1}}\\
\coolrightbrace{y \\ y }{a_{\ell -1}}
\end{matrix*}\]

Note that for any of the blocks $\Gamma_{j,k}$ in (\ref{blocksdiag}), since $\ell \leq \delta-i \leq i$, and $\ell \leq j \leq i$, it follows that $a_{\ell} \leq a_{j}$. Also, since $\delta -i \leq k \leq \delta -\ell$, we have $a_{\ell} = a_{\delta -\ell} \leq a_k$. Thus, we indeed have enough entries to pick in all blocks.

Furthermore, for a group of rows corresponding to $B_j$, we picked entries from the bottom $a_0$ rows of the block $\Gamma_{j,\delta+j-i}$, then entries from the next $a_1-a_0$ rows from the block $\Gamma_{j,\delta+j-i-1}$, and so on, so that we never select entries from the same rows. The same reasoning applies to columns. Fixing a group of columns corresponding to $B_k$, we pick $a_0$ entries from right column of $\Gamma_{k+i-\delta, k}$, then $a_1-a_0$ entries from the next columns, and so on, never repeating columns. So we select a single entry from each row and each column.

At each step $\ell$, for $0 \leq \ell \leq \delta-i$, we picked $a_{\ell} - a_{\ell-1}$ entries of the form $c_{\delta-\ell-i}$ from $i-\ell+1$ blocks. Taking the product of all entries selected, we have a polynomial in the coefficients of $g$ of the form
$$c_{\delta-i}^{(i+1)a_0} c_{\delta-i-1}^{i(a_1-a_0)} c_{\delta-i-2}^{(i-1)(a_2-a_1)} \cdots c_{0}^{(2i-\delta+1)(a_{\delta-i}-a_{\delta-i-1})} + \mbox{ other terms}. $$
\end{proof}

\begin{lemma} \label{lemmaunisel}
There is only one way of selecting entries from $\Theta _i$ and obtaining the term in Equation (\ref{term1}).
\end{lemma}

\begin{proof}
We use induction to show that for $\ell = \delta-i, \delta-i-1, \dots, 0$, there is only one way of obtaining the power of $c_{\ell}$ in Equation (\ref{term1}) from the product of entries of $\Theta_i$.

We first consider $c_{\delta-i}$. We now use induction to show that for all $0 \leq j \leq i$, the only entry available to select in the last row of the set of rows corresponding to $B_j$ is the one on the last column of the block $\Gamma_{j,\delta+j-i}$, of the form $c_{\delta-i} + L$. Note that the only coefficient in Equation (\ref{term1}) that appears in the bottom row of $\Theta_i$, corresponding to $B_0$, is $c_{\delta-i}$, which is in the last column of the block $\Gamma_{0,\delta-i}$. So we pick this entry. Suppose now that the only way of selecting an entry containing a coefficient in Equation (\ref{term1}) from the last row of the block corresponding to $B_j$ is picking the one containing $c_{\delta-i}$, from the last column of the block $\Gamma_{j, \delta+j-i}$. This means that the other entries in this row involving coefficients in Equation (\ref{term1}) cannot be selected at this point, and so entries in the last columns of the blocks $\Gamma_{j, \delta+j-i-1}, \Gamma_{j, \delta+j-i-2}, \cdots$ have been selected in previous steps, from blocks below. Passing to the set of rows corresponding to $B_{j+1}$, it follows that entries have been selected on the last columns of blocks $\Gamma_{j+1,\delta+j-i}, \Gamma_{j+1, \delta+j-i-1, \dots}$, and hence we are left with no choice other than selecting the entry from the last column of the block $\Gamma_{j+1, \delta+j-i+1}$, which has the form $c_{\delta-i}$. This proves our claim.

Let $1\leq \ell \leq \delta-i$, and suppose we have selected entries involving $c_{\delta-i}, c_{\delta-i-1}, \dots, c_{\delta-i-\ell+1}$ as in Lemma \ref{lemmaselec}, and that this selection was the only possible choice. This means that we have already picked entries from the bottom $a_{\ell-1}$ rows of all blocks $B_0, \dots, B_i$. So let us consider the next $a_{\ell} - a_{\ell -1}$ rows. Starting with the block $B_{\ell}$, note that the coefficients from Equation (\ref{term1}) that appear in this block are $c_{\delta-i}, c_{\delta-i-1}, \dots, c_{\delta-i-\ell}$. But with the selections we have already made, the exponents of $c_{\delta-i}, c_{\delta-i-1}, \dots, c_{\delta-i-\ell+1}$ in Equation (\ref{term1}) were reached, so that at this point we cannot select the entries involving these coefficients. Hence, the only choice left is selecting the entries of the form $c_{\delta-i-\ell} + L$ from $\Gamma_{\ell, \delta-i}$.

Let $\ell+1 \leq j \leq i$. Suppose we already picked entries of the form $c_{\delta-i-\ell} + L$ as in Lemma \ref{lemmaselec} from blocks $B_{\ell}, B_{\ell+1}, \dots, B_{j-1}$. Consider block $B_j$. Still assuming that entries have been selected from the bottom $a_{\ell-1}$ rows, we pass to the next $a_{\ell} - a_{\ell-1}$. The selections made in blocks bellow prevent us from picking the entries involving $c_{\delta-i-\ell-1}, \dots, c_0$. Also, we cannot select entries where coefficients $c_{\delta-i}, \dots, c_{\delta-i-\ell+1}$ appear. Thus, we are left with entries containing $c_{\delta-i-\ell}$.
\end{proof}

\begin{lemma}\label{lemmaselec2}
Suppose $\delta - d \leq i \leq \delta$, and $i \leq \delta/2$. Then the term
\begin{equation}\label{term2} c_{\delta-i}^{(i+1)a_{0}} c_{\delta-i-1}^{i(a_1-a_0)} c_{\delta-i-2}^{(i-1)(a_2-a_1)} \cdots c_{i}^{(a_{i}-a_{i-1})} \end{equation}
can be obtained from the product of entries of $\Theta_i$, with exactly one entry from each column and row.
\end{lemma}

\begin{proof}
The proof is the same as Lemma \ref{lemmaselec}, except that in this case we select entries in steps $\ell$, for $0 \leq \ell \leq \delta-2i$.
\end{proof}

The same proof of Lemma \ref{lemmaunisel} works to show the following.

\begin{lemma}\label{lemmaunisel1}
There is only one way of selecting entries from $\Theta _i$ and obtaining the term in Equation (\ref{term2}).
\end{lemma}

\begin{corollary}\label{dmaiorquedelta} $\det \Theta_i \neq 0$ for $\delta-d \leq i \leq \delta$.
\end{corollary}

When $d \geq \delta$, Corollary \ref{dmaiorquedelta} holds for all matrices $\Theta_i$, and we have the following.

\begin{proposition} If $d \geq \delta$, then
$$\ini{I,g} = ( \ini{I}, z^{d-\delta}B_{\delta}, z^{d-\delta+2}B_{\delta-1}, \dots, z^{\delta+d-3}B_1, z^{\delta+d-1}B_0).$$
\end{proposition}
\begin{proof}
Fix $1 \leq i \leq \delta$. Since $\Theta_i$ is the submatrix of $M_i$ formed by columns corresponding to the monomials in $B_{\delta}z^{d+i-\delta}, \dots, B_{\delta-i} z^{d+2i-\delta}$, and by Corollary \ref{dmaiorquedelta} is nonsingular, we can perform row operations on $M_i$ and change Equation (\ref{matrixmi}) into
\begin{equation*}
\begin{pmatrix}
  \mathbf{u}_i \\
  \mathbf{u}_{i-1} \\
  \vdots \\
  \mathbf{u}_0 \\
\end{pmatrix}
\cdot g \equiv \begin{pmatrix}
  \mathbf{w}_i \\
  \mathbf{w}_{i-1} \\
  \vdots \\
  \mathbf{w}_0 \\
\end{pmatrix}
\mo G,
\end{equation*}
where the entries of each $\mathbf{w}_j$ are polynomials with distinct initial terms, so that each monomial in $B_{\delta}z^{d+i-\delta}, \dots, B_{\delta-i} z^{d+2i-\delta}$ occurs as leading monomial of some polynomial in $\mathbf{w}_0, \dots, \mathbf{w}_i$. But the monomials in $B_{\delta}z^{d+i-\delta}$, $\dots$, $B_{\delta-i+1} z^{d+2(i-1)-\delta}$ are redundant as they are multiples of monomials that occur as leading terms when we perform row operations on $\mathbf{E}_{i-1} g \equiv M_{i-1}\mathbf{E}_{i+d-1} \mo G$. Thus, only the monomials in $B_{\delta-i} z^{d+2i-\delta}$ are minimal generators of $\ini{I,g}$.
\end{proof}

\begin{corollary} Suppose $d \geq \delta$, and let $\widetilde{B} \subset K[x_1, \dots, x_n, z]$ denote the set of standard monomials of $(I,g)$. Then
\begin{eqnarray*}
\widetilde{B}_0 & = & B_0 \\
\widetilde{B}_1 & = & B_1 \cup zB_0 \\
\widetilde{B}_2 & = & B_2 \cup zB_1 \cup z^2B_0 \\
& \vdots & \\
\widetilde{B}_{\delta} & = & B_{\delta} \cup zB_{\delta-1} \cup \cdots \cup z^{\delta}B_0 \\
\widetilde{B}_{\delta+1} & = & z\widetilde{B}_{\delta} \\
& \vdots & \\
\widetilde{B}_{d-1} & = & z^{d-\delta-1}\widetilde{B}_{\delta} \\
\widetilde{B}_d & = & z^{d-\delta+1}\widetilde{B}_{\delta-1} \\
\widetilde{B}_{d+1} & = & z^{d-\delta+3} \widetilde{B}_{\delta-2} \\
& \vdots & \\
\widetilde{B}_{d+\delta-1} & = & z^{d+\delta-1}B_0.
\end{eqnarray*}
\end{corollary}

We now turn to the case $d < \delta$ and $0 \leq i < \delta-d$. In this case, it is clear that the monomials of degree $d+i$ that enter the basis of $\ini{I,g}$ are not necessarily the largest monomials in $E_{i+d}$, and hence the square submatrix of $M_i$ formed by columns corresponding to those monomials is not necessarily nonsingular. In fact, the blocks $\Gamma_{j,k}$ for $0 \leq j <i$ and $d < k \leq d+i$ have all entries equal to zero, thus the matrix formed by columns corresponding to the greatest monomials might have rows of zeroes (it is certainly the case for $d < \delta -2$).

However, for the Moreno-Socías Conjecture to be true, what we need is the greatest \emph{non-redundant} monomials to enter the basis at each step. Let $i^{\star} = \lfloor\frac{\delta-d}{2}\rfloor$. We conjecture the following.

\begin{conjecture}\label{matrixismall}
Suppose $d_1 \leq d_2 \leq \cdots \leq d_n \leq d$, and $0 \leq i \leq i^{\star}$. Let $\Theta_i$ denote the square submatrix of $M_i$ formed by the columns corresponding to the $a_i$ largest monomials of $B_{i+d}$, the $a_{i-1}$ largest monomials of $zB_{i+d-1}$, and so on, up to the $a_0$ largest monomials of $z^{i}B_d$. Then $\Theta_i$ is nonsingular.
\end{conjecture}

\begin{conjecture}\label{matriximedium}
Suppose $d_1 \leq d_2 \leq \cdots \leq d_n \leq d$, and $i^{\star} < i < \delta-d$. Let $\Theta_i$ denote the square submatrix of $M_i$ formed by columns corresponding to
\begin{enumerate}[{\normalfont (i)}]
\item all monomials in $B_{d+j}$, for $\delta-d-i \leq j \leq i$, and
\item the $a_j$ largest monomials in $B_{d+j}$, for $0 \leq j < \delta-d-i$.
\end{enumerate}
Then $\Theta_i$ is nonsingular.
\end{conjecture}

In fact, matrices $\Theta_j$ for $0 \leq j \leq i^{\star}$ are submatrices of $\Theta_i$ for $i^{\star}< i <\delta-d$, and if we can prove the smaller matrices are nonsingular, we are actually able to prove all $\Theta_i$ are nonsingular.

\begin{proposition} Conjecture \ref{matrixismall} implies Conjecture \ref{matriximedium}. \end{proposition}
\begin{proof}
Let $i^{\star} < i < \delta-d$. Let $\Lambda_i$ denote the submatrix of $\Theta_i$ formed by the following blocks
\begin{equation*}
\Lambda_i = \begin{pmatrix}
              \Gamma_{i,d+i} & \Gamma_{i,d+i-1} & \cdots & \Gamma_{i,\delta-i} \\
              \Gamma_{i-1,d+i} & \Gamma_{i-1,d+i-1} & \cdots & \Gamma_{i-1,\delta-i} \\
              \vdots & \vdots & \ddots & \vdots \\
              \Gamma_{\delta-d-i,d+i} & \Gamma_{\delta-d-i,d+i-1} & \cdots & \Gamma_{\delta-d-i,\delta-i} \\
            \end{pmatrix}
\end{equation*}
Then, $\Theta_i$ can be written as
\begin{equation*}
\Theta_i = \left( \begin{array}{c|c}
\Lambda_i & \Omega \\ \hline 0 & \Theta_{\delta-d-i-1}
\end{array}
\right)
\end{equation*}
that is, the columns formed by $ {\Lambda_i \choose 0} $ are the ones in Conjecture \ref{matriximedium}(i), and the columns formed by ${\Omega} \choose \Theta_{\delta-d-i-1}$ are the columns in (ii).

Now, $\det \Theta_i = \det (\Lambda_i) \cdot \det (\Theta_{\delta-d-i-1})$. If Conjecture \ref{matrixismall} is true, then $\det \Theta_{\delta-d-i-1} \neq 0$. So we need to see that $\det \Lambda_i \neq 0$. In fact, an argument similar to that applied in Lemmas \ref{lemmaselec}-\ref{lemmaunisel1} can be used. We claim the term
\begin{equation}\label{term3}
c_d^{(2i+d-\delta+1)a_{d+i}} c_{d-1}^{(2i+d-\delta)(a_{d+i-1}-a_{d+i})} \cdots c_{\delta-2i}^{(a_{\delta-i}-a_{\delta-i+1})}
\end{equation}
appears in the determinant of $\Lambda_i$. Again we start by selecting entries from the blocks on the diagonal at step 0, and then from the blocks above the diagonal at step 1, and so on.

In general, at step $\ell$, for $0 \leq \ell \leq 2i+d-\delta$, we select entries from the blocks
$$ \Gamma_{i,d_i-\ell}, \Gamma_{i-1, d+i-\ell-1}, \dots, \Gamma_{\delta-d-i+\ell, \delta-i}.$$

We select the entries in the diagonal of the bottom $a_{d+i-\ell}$ rows and right-most $a_{d+i-\ell}$ columns, skipping the bottom $a_{d+i-\ell+1}$. The proof that these selections can be made, and that this is the only way of obtaining the term (\ref{term3}) is identical to Lemma \ref{lemmaselec} and Lemma \ref{lemmaunisel}.
\end{proof}

The condition $d_1 \leq d_2 \leq \cdots \leq d_n \leq d$ is necessary. When the degrees are not in this order, our description fails, as we can see from the next example.

\begin{example}
Let $I$ be the ideal in $K[x_1, x_2, x_3, z]$ generated by generic forms of degree $4$, $4$ and $5$, and let $g$ be a generic polynomial of degree $3$. According to Conjecture \ref{matrixismall}, we consider the matrix $\Theta_3$ formed by the $a_3 = 10$ greatest monomials in $B_6$, $a_2 = 6$ greatest monomials in $zB_5$, $a_1 =3$ greatest monomials in $z^2B_4$, and $a_0 =1$ greatest monomial in $z^3B_3$. This matrix, however, is singular, as the row corresponding to $x_3z^2 \cdot g$ is zero. In fact, the monomial $x_2^4z^2$, which is the third greatest monomial in $z^2B_4$, is not in the basis of $\ini{I,g}$.
\end{example}

When $d = \delta-1$, the only matrix treated in Conjecture \ref{matrixismall} is $\Theta_0$, which is a one by one matrix whose single entry is the leading coefficient of $g$, and thus is nonzero. For $d = \delta-2$, $\Theta_0$ is once again a one by one matrix whose entry is $\lc{g}$, and $\Theta_1$ is given by
\begin{equation*}
\Theta_1 = \left(
\begin{array}{c|c}
\Gamma_{1,\delta-1} & \Omega \\ \hline
0 & \lc{g}
\end{array}
\right) = \left(
\begin{array}{cccc|c}
c_{\delta-2}+L & * & \cdots & * & * \\
* & c_{\delta-2}+L & \cdots & * & * \\
 \vdots & \vdots & \ddots & \vdots & \vdots \\
* & * &  \cdots & c_{\delta-2}+L & * \\ \hline
0 & 0 & \cdots & 0 & \lc{g}
\end{array}
\right)
\end{equation*}
so $\det \Theta_1 = \lc{g} \cdot \det \Gamma_{1,\delta-1}$, and the determinant of $\Gamma_{1,\delta-1}$ is nonzero because the term $c_{\delta-2}^{a_1}$ appears in it. This, together with the results from the previous section, proves the following.

\begin{proposition}\label{inarl}
Suppose $d \geq \delta-2$. If $\ini{I}$ is almost reverse lexicographic, then $\ini{I,g}$ is almost reverse lexicographic.
\end{proposition}
\begin{proof}
For all $t\geq d$, the minimal generators of degree $t$ introduced to the basis of $\ini{I,g}$ are the largest monomials in $E_t$.
\end{proof}

Using induction we have a partial answer to Moreno-Socías Conjecture.

\begin{theorem}\label{morenosociaspartial}
Let $I = ( f_1, \dots, f_n) \subset K\nvar$ be a generic ideal, with $\deg(f_i) = d_i$ and $d_i \geq \left(\sum_{j=1}^{i-1} d_j\right) -i -1$. Then $\ini{I}$ is almost reverse lexicographic.
\end{theorem}
\begin{proof}
The result clearly holds for $n=1$. Assuming it holds for $n-1$, the initial ideal of $( f_1, \dots, f_{n-1})$ $\subset K[x_1, \dots, x_{n-1}]$ is almost reverse lexicographic. By Proposition \ref{inarl}, $\ini{I}$ is almost reverse lexicographic.
\end{proof}

The theorem above is somewhat more general than the result given in \cite{MR2397406}, where Cho and Park proved the case $d_i > \left(\sum_{j=1}^{i-1} d_j\right) -i +1$. We believe that our approach is promising, and that by investigating further the properties of $\mathrm{B}(I)$ and the structure of the matrices from Conjecture \ref{matrixismall}, we could be able to give an answer to Moreno-Socías Conjecture.

If Conjecture \ref{matrixismall} is true, then we can give a description of $\ini{I,g}$ as follows. We will use the following notation: for a set $S = \{s_1, \dots, s_{\ell}\}$ and $1 \leq a \leq b \leq \ell$, let
\begin{eqnarray*}
S^{[a,b]} & = & \{ s_a, \dots, s_b\}, \\
S^{(a,b]} & = & \{s_{a+1}, \dots, s_b\}.
\end{eqnarray*}
If $\delta-d \equiv 0 \mo 2$, the initial ideal of $( I,g )$ can be described as
\begin{eqnarray*}
\ini{I,g} & = ( & \ini{I}, B_{d}^{[1,a_0]}, B_{d+1}^{[1,a_1]}, \dots, B_{d+i^{\star}-1}^{[1,a_{i^{\star}-1}]}, B_{d+i^{\star}}, \\
& & z^2B_{d+i^{\star}-1}^{(a_{i^{\star}-1},a_{d+i^{\star}-1}]}, z^4B_{d+i^{\star}-2}^{(a_{i^{\star}-2},a_{d+i^{\star}-2}]}, \dots, z^{\delta-d}B_d^{(a_0,a_d]}, \\
& & z^{\delta-d+2}B_{d-1}, \dots, z^{\delta+d-2}B_1, z^{\delta+d}B_0 \ \ \ ).
\end{eqnarray*}

The corresponding set of standard monomials $\widetilde{B}$ is
\begin{eqnarray*}
\widetilde{B}_0 & = & B_0 \\
\widetilde{B}_1 & = & B_1 \cup zB_0 \\
\widetilde{B}_2 & = & B_2 \cup zB_1 \cup z^2B_0 \\
& \vdots & \\
\widetilde{B}_{d-1} & = & B_{d-1} \cup z\widetilde{B}_{d-2} \\
\widetilde{B}_d & = & B_d^{(a_0,a_d]} \cup z\widetilde{B}_{d-1} \\
\widetilde{B}_{d+1} & = & B_{d+1}^{(a_1,a_{d+1}]} \cup z\widetilde{B}_{d-2} \\
& \vdots & \\
\widetilde{B}_{d+i^{\star}} & = & z\widetilde{B}_{d+i^{\star}-1} \\
\widetilde{B}_{d+i^{\star}+1} & = & z^3\widetilde{B}_{d+i^{\star}-2} \\
& \vdots & \\
\widetilde{B}_{\delta} & = & z^{\delta-d+1}\widetilde{B}_{d-1} \\
\widetilde{B}_{\delta+1} & = & z^{\delta-d+3}\widetilde{B}_{d-2} \\
& \vdots & \\
\widetilde{B}_{\delta+d-1} & = & z^{\delta+d-1}\widetilde{B}_0.
\end{eqnarray*}

If $\delta-d \equiv 1 \mo 2$, the initial ideal of $( I,g )$ can be described as
\begin{eqnarray*}
\ini{I,g} & = ( & \ini{I}, B_{d}^{[1,a_0]}, B_{d+1}^{[1,a_1]}, \dots, B_{d+i^{\star}}^{[1,a_{i^{\star}}]}, \\
& & zB_{d+i^{\star}}^{(a_{i^{\star}},a_{d+i^{\star}}]}, z^3B_{d+i^{\star}-1}^{(a_{i^{\star}-1},a_{d+i^{\star}-1}]}, \dots, z^{\delta-d}B_d^{(a_0,a_d]}, \\
& & z^{\delta-d+2}B_{d-1}, \dots, z^{\delta+d-2}B_1, z^{\delta+d}B_0 \ \ \ ).
\end{eqnarray*}

The corresponding set of standard monomials $\widetilde{B}$ is
\begin{eqnarray*}
\widetilde{B}_0 & = & B_0 \\
\widetilde{B}_1 & = & B_1 \cup zB_0 \\
\widetilde{B}_2 & = & B_2 \cup zB_1 \cup z^2B_0 \\
& \vdots & \\
\widetilde{B}_{d-1} & = & B_{d-1} \cup z\widetilde{B}_{d-2} \\
\widetilde{B}_d & = & B_d^{(a_0,a_d]} \cup z\widetilde{B}_{d-1} \\
\widetilde{B}_{d+1} & = & B_{d+1}^{(a_1,a_{d+1}]} \cup z\widetilde{B}_{d-2} \\
& \vdots & \\
\widetilde{B}_{d+i^{\star}} & = & B_{d+i^{\star}}^{(a_{i^{\star}}, a_{d+i^{\star}}]} \cup z\widetilde{B}_{d+i^{\star}-1} \\
\widetilde{B}_{d+i^{\star}+1} & = & z^2\widetilde{B}_{d+i^{\star}-1} \\
\widetilde{B}_{d+i^{\star}+2} & = & z^4\widetilde{B}_{d+i^{\star}-2} \\
& \vdots & \\
\widetilde{B}_{\delta} & = & z^{\delta-d+1}\widetilde{B}_{d-1} \\
\widetilde{B}_{\delta+1} & = & z^{\delta-d+3}\widetilde{B}_{d-2} \\
& \vdots & \\
\widetilde{B}_{\delta+d-1} & = & z^{\delta+d-1}\widetilde{B}_0.
\end{eqnarray*}

From the description above, we have that Conjecture \ref{matrixismall} implies that $\ini{I,g}$ is almost reverse lexicographic, and hence also implies the Moreno-Socías conjecture.

\begin{example}
Let $f_1, f_2$ be generic polynomials of degrees $d_1 = d_2 = 4$, and let $I = ( f_1, f_2 )$. The initial ideal of $I$ is given by
$$ \ini{I} =( x_1^4, x_1^3x_2, x_1^2x_2^3, x_1x_2^5, x_2^7 ). $$
Then $\delta=6$, and we consider $g$ of degree $d = 4 = \delta -2$.
 \begin{eqnarray*}
 g & = & b_1 x_1^2x_2^2 + b_2 x_1x_2^3 + b_3 x_2^4 + b_4x_1^3z + b_5x_1^2x_2z + b_6x_1x_2^2z + b_7x_2^3z + b_8x_1^2z^2 + b_9x_1x_2z^2  \\
 & & + b_{10}x_2^2z^2 + b_{11}x_1z^3 + b_{12}x_2z^3 + b_{13}z^4.
 \end{eqnarray*}

 We give the matrices $\Theta_i$ below. We write entries as functions of the coefficients $b_i$'s. All entries have the form $b_i + L(b_1, \dots, b_{i-1})$ or $L(b_1, \dots, b_i)$. We show only the entries of the first form, ignoring the $L$ portion. The entries selected to form the terms in Lemma \ref{lemmaselec} and Lemma \ref{lemmaselec2} are shown in boldface.  We start with $\Theta_{6} = M_6$:
\begin{equation*}
\arraycolsep=1.4pt\def\arraystretch{0.8}
\kbordermatrix{
 & x_2^6 & x_1x_2^4 & x_2^5 & x_1^2x_2^2 & x_1x_2^3 & x_2^4 & x_1^3 & x_1^2x_2 & x_1x_2^2 & x_2^3 & x_1^2 & x_1x_2 & x_2^2 & x_1 & x_2 & 1 \\
 x_2^6 & \mathbf{b_{13}} &  & & & & & & & & & & & & & & \\
x_1x_2^4 & & \mathbf{b_{13}} & & & & & & & & & & & & & & \\
x_2^5 & b_{12} &  & \mathbf{b_{13}} & & & & & & & & & & & & & \\
x_1^2x_2^2 &  &  &  & \mathbf{b_{13}} & & & & & & & & & & & & \\
x_1x_2^3 &  & b_{12} &  &  & \mathbf{b_{13}} & & & & & & & & & & & \\
x_2^4 & b_{10} & b_{11} & b_{12} &  &  & \mathbf{b_{13}} & & & & & & & & & & \\
x_1^3 & &  & & & & & \mathbf{b_{13}} & & & & & & & & & \\
x_1^2x_2 & &  & & b_{12} & & &  & \mathbf{b_{13}} & & & & & & & & \\
x_1x_2^2 & & b_{10} & & b_{11} & b_{12} & &  &  & \mathbf{b_{13}} & & & & & & & \\
x_2^3 & b_7 & b_{9} & b_{10} &  & b_{11} & b_{12} &  &  &  & \mathbf{b_{13}} & & & & & & \\
x_1^2 &  &  &  & b_{10} &  &  & b_{11} & b_{12} &  & & \mathbf{b_{13}} & & & & & \\
x_1x_2 &  & b_7 &  & b_{9} & b_{10} &  & & b_{11} & b_{12}  & &  & \mathbf{b_{13}} & & & & \\
x_2^2 & b_3 & b_6 & b_7 & b_{8} & b_{9} & b_{10} & & & b_{11} & b_{12} &  & & \mathbf{b_{13}} & & & \\
x_1 &  & b_3 & & b_{6} & b_{7} &  & & b_9 & b_{10} &  & b_{11} & b_{12} &  & \mathbf{b_{13}} & & \\
x_2 &  & b_2 & b_3 & b_{5} & b_{6} & b_7 & & b_8 & b_{9} & b_{10} & & b_{11} & b_{12} & & \mathbf{b_{13}} & \\
1 &  & & & b_1 & b_{2} & b_3 & b_4 & b_5 & b_{6} & b_{7} & b_8 & b_{9} & b_{10} & b_{11} & b_{12} & \mathbf{b_{13}} \\
}.
\end{equation*}
The entries in boldface give a nonzero term in $\det \Theta_6$. Since the determinant is nonzero, performing row operations on
$$ \mathbf{E}_6 \cdot g \equiv M_6 \mathbf{E}_{10} \mo G,$$
all monomials in $E_{10}$ will appear as leading monomials on the right-hand side. Thus, the monomials
$$ x_2^6z^4, x_1x_2^4z^5, x_2^5z^5, x_1^2x_2^2z^6, x_1x_2^3z^6, x_2^4z^6, x_1^3z^7, x_1^2x_2z^7, x_1x_2^2z^7, x_2^3z^7, x_1^2z^8, x_1x_2z^8, x_2^2z^8, x_1z^9, x_2z^9, z^{10}$$
are in the basis of $\ini{I,g}$. The matrix $\Theta_5$ is obtained from $\Theta_6$ by removing the top row and right-most column. Again we show in boldface the entries that are used to guarantee that the determinant of this matrix is nonzero. This is the form of $\Theta_5$
\begin{equation*}
\arraycolsep=1.4pt\def\arraystretch{0.8}
\kbordermatrix{
 & x_2^6 & x_1x_2^4 & x_2^5 & x_1^2x_2^2 & x_1x_2^3 & x_2^4 & x_1^3 & x_1^2x_2 & x_1x_2^2 & x_2^3 & x_1^2 & x_1x_2 & x_2^2 & x_1 & x_2 \\
x_1x_2^4 & & \mathbf{b_{13}} & & & & & & & & & & & & & \\
x_2^5 & \mathbf{b_{12}} &  & {b_{13}} & & & & & & & & & & & &  \\
x_1^2x_2^2 &  &  &  & \mathbf{b_{13}} & & & & & & & & & & &  \\
x_1x_2^3 &  & b_{12} &  &  & \mathbf{b_{13}} & & & & & & & & & &  \\
x_2^4 & b_{10} & b_{11} & \mathbf{b_{12}} &  &  & {b_{13}} & & & & & & & & &  \\
x_1^3 & &  & & & & & \mathbf{b_{13}} & & & & & & & &  \\
x_1^2x_2 & &  & & b_{12} & & &  & \mathbf{b_{13}} & & & & & & &  \\
x_1x_2^2 & & b_{10} & & b_{11} & b_{12} & &  &  & \mathbf{b_{13}} & & & & & &  \\
x_2^3 & b_7 & b_{9} & b_{10} &  & b_{11} & \mathbf{b_{12}} &  &  &  & {b_{13}} & & & & & \\
x_1^2 &  &  &  & b_{10} &  &  & b_{11} & b_{12} &  & & \mathbf{b_{13}} & & & &  \\
x_1x_2 &  & b_7 &  & b_{9} & b_{10} &  & & b_{11} & b_{12}  & &  & \mathbf{b_{13}} & & &  \\
x_2^2 & b_3 & b_6 & b_7 & b_{8} & b_{9} & b_{10} & & & b_{11} & \mathbf{b_{12}} &  & & {b_{13}} & &  \\
x_1 &  & b_3 & & b_{6} & b_{7} &  & & b_9 & b_{10} &  & b_{11} & b_{12} &  & \mathbf{b_{13}} &  \\
x_2 &  & b_2 & b_3 & b_{5} & b_{6} & b_7 & & b_8 & b_{9} & b_{10} & & b_{11} & \mathbf{b_{12}} & & {b_{13}}  \\
1 &  & & & b_1 & b_{2} & b_3 & b_4 & b_5 & b_{6} & b_{7} & b_8 & b_{9} & b_{10} & b_{11} & \mathbf{b_{12}} \\
}.
\end{equation*}
So, performing row operations on
$$\mathbf{E}_5 \cdot g \equiv M_5 \mathbf{E}_9 \mo G$$
leads to the 15 greatest monomials in $E_{9}$ being leading monomials on the right-hand side, which means that
$$ x_2^6z^3, x_1x_2^4z^4, x_2^5z^4, x_1^2x_2^2z^5, x_1x_2^3z^5, x_2^4z^5, x_1^3z^6, x_1^2x_2z^6, x_1x_2^2z^6, x_2^3z^6, x_1^2z^7, x_1x_2z^7, x_2^2z^7, x_1z^8, x_2z^8$$
are in $\ini{I,g}$. Next, we consider
$$ \mathbf{E}_4 \cdot g \equiv M_4 \mathbf{E}_8 \mo G. $$
The matrix $M_4$ is $13 \times 16$, and $\Theta_4$ is the $13 \times 13$ submatrix given by
\begin{equation*}
\arraycolsep=1.4pt\def\arraystretch{0.8}
\kbordermatrix{
 & x_2^6 & x_1x_2^4 & x_2^5 & x_1^2x_2^2 & x_1x_2^3 & x_2^4 & x_1^3 & x_1^2x_2 & x_1x_2^2 & x_2^3 & x_1^2 & x_1x_2 & x_2^2  \\
x_1^2x_2^2 &  &  &  & \mathbf{b_{13}} & & & & & & & & &  \\
x_1x_2^3 &  & \mathbf{b_{12}} &  &  & {b_{13}} & & & & & & & &  \\
x_2^4 & \mathbf{b_{10}} & b_{11} & {b_{12}} &  &  & {b_{13}} & & & & & & &   \\
x_1^3 & &  & & & & & \mathbf{b_{13}} & & & & & &   \\
x_1^2x_2 & &  & & b_{12} & & &  & \mathbf{b_{13}} & & & & &  \\
x_1x_2^2 & & b_{10} & & b_{11} & \mathbf{b_{12}} & &  &  & {b_{13}} & & & &  \\
x_2^3 & b_7 & b_{9} & \mathbf{b_{10}} &  & b_{11} & {b_{12}} &  &  &  & {b_{13}} & & &\\
x_1^2 &  &  &  & b_{10} &  &  & b_{11} & b_{12} &  & & \mathbf{b_{13}} & &  \\
x_1x_2 &  & b_7 &  & b_{9} & b_{10} &  & & b_{11} & \mathbf{b_{12}}  & &  & {b_{13}} &   \\
x_2^2 & b_3 & b_6 & b_7 & b_{8} & b_{9} & \mathbf{b_{10}} & & & b_{11} & {b_{12}} &  & & {b_{13}}   \\
x_1 &  & b_3 & & b_{6} & b_{7} &  & & b_9 & b_{10} &  & b_{11} & \mathbf{b_{12}} &   \\
x_2 &  & b_2 & b_3 & b_{5} & b_{6} & b_7 & & b_8 & b_{9} & \mathbf{b_{10}} & & b_{11} & {b_{12}}   \\
1 &  & & & b_1 & b_{2} & b_3 & b_4 & b_5 & b_{6} & b_{7} & b_8 & b_{9} & \mathbf{b_{10}} \\
},
\end{equation*}
which is also a submatrix of $\Theta_5$, obtained by removing the rows corresponding to $B_5$ and the columns corresponding to $B_1$. After row operations,
$$ x_2^6z^2, x_1x_2^4z^3, x_2^5z^3, x_1^2x_2^2z^4, x_1x_2^3z^4, x_2^4z^4, x_1^3z^5, x_1^2x_2z^5, x_1x_2^2z^5, x_2^3z^5, x_1^2z^6, x_1x_2z^6, x_2^2z^6$$
are leading monomials. Similarly, removing from $\Theta_4$ the rows corresponding to $B_4$ and the columns corresponding to $B_2$, we get $\Theta_3$ given by
\begin{equation*}
\arraycolsep=1.4pt\def\arraystretch{0.8}
\kbordermatrix{
 & x_2^6 & x_1x_2^4 & x_2^5 & x_1^2x_2^2 & x_1x_2^3 & x_2^4 & x_1^3 & x_1^2x_2 & x_1x_2^2 & x_2^3 \\
x_1^3 & &  & & & & & \mathbf{b_{13}} & & & \\
x_1^2x_2 & &  & & \mathbf{b_{12}} & & &  & {b_{13}} & & \\
x_1x_2^2 & & \mathbf{b_{10}} & & b_{11} & {b_{12}} & &  &  & {b_{13}} &  \\
x_2^3 & \mathbf{b_7} & b_{9} & {b_{10}} &  & b_{11} & {b_{12}} &  &  &  & {b_{13}} \\
x_1^2 &  &  &  & b_{10} &  &  & b_{11} & \mathbf{b_{12}} &  & \\
x_1x_2 &  & b_7 &  & b_{9} & \mathbf{b_{10}} &  & & b_{11} & {b_{12}}  &  \\
x_2^2 & b_3 & b_6 & \mathbf{b_7} & b_{8} & b_{9} & {b_{10}} & & & b_{11} & {b_{12}}    \\
x_1 &  & b_3 & & b_{6} & b_{7} &  & & b_9 & \mathbf{b_{10}} &    \\
x_2 &  & b_2 & b_3 & b_{5} & b_{6} & \mathbf{b_7} & & b_8 & b_{9} & {b_{10}}   \\
1 &  & & & b_1 & b_{2} & b_3 & b_4 & b_5 & b_{6} & \mathbf{b_{7}} \\
}.
\end{equation*}
The leading monomials obtained are
$$ x_2^6z, x_1x_2^4z^2, x_2^5z^2, x_1^2x_2^2z^3, x_1x_2^3z^3, x_2^4z^3, x_1^3z^4, x_1^2x_2z^4, x_1x_2^2z^4, x_2^3z^4.$$
Next, $\Theta_2$ is given by
\begin{equation*}
\arraycolsep=1.4pt\def\arraystretch{0.8}
\kbordermatrix{
 & x_2^6 & x_1x_2^4 & x_2^5 & x_1^2x_2^2 & x_1x_2^3 & x_2^4  \\
x_1^2 &  &  &  & \mathbf{b_{10}} &  &   \\
x_1x_2 &  & \mathbf{b_7} &  & b_{9} & {b_{10}} &   \\
x_2^2 & \mathbf{b_3} & b_6 & {b_7} & b_{8} & b_{9} & {b_{10}} \\
x_1 &  & b_3 & & b_{6} & \mathbf{b_{7}} &   \\
x_2 &  & b_2 & \mathbf{b_3} & b_{5} & b_{6} & {b_7} \\
1 &  & & & b_1 & b_{2} & \mathbf{b_3} \\
},
\end{equation*}
and the elements in $E_6$ that are leading monomials are
$$x_2^6, x_1x_2^4z, x_2^5z, x_1^2x_2^2z^2, x_1x_2^3z^2, x_2^4z^2.$$
The matrix $\Theta_1$ is given by
\begin{equation*}
\arraycolsep=1.4pt\def\arraystretch{0.8}
\kbordermatrix{
  & x_1x_2^4 & x_2^5 & x_1^2x_2^2 \\
x_1 &  \mathbf{b_3} & & b_{6}  \\
x_2 &  b_2 & \mathbf{b_3} & b_{5}  \\
1 &  & & \mathbf{b_1} \\
},
\end{equation*}
and the monomials of degree 5 that enter the basis of $\ini{I,g}$ are
$$ x_1x_2^4, x_2^5, x_1^2x_2^2z. $$
Finally, $\Theta_0$ is the $1 \times 1$ matrix
\begin{equation*}
\arraycolsep=1.4pt\def\arraystretch{0.8}
\kbordermatrix{
  & x_1^2x_2^2 \\
1 & \mathbf{b_1} \\
},
\end{equation*}
and the monomial $x_1^2x_2^2$ is in $\ini{I,g}$. Putting all the leading monomials we found together, and discarding the redundant ones, we have
\begin{eqnarray*}
\ini{I,g} & = & ( x_1^4, x_1^3x_2, x_1^2x_2^2, x_1x_2^4, x_2^5, x_1x_2^3z^2, x_2^4z^2, x_1^3z^4, x_1^2x_2z^4, \\
& & x_1x_2^2z^4, x_2^3z^4, x_1^2z^6, x_1x_2z^6, x_2^2z^6, x_1z^8, x_2z^8, z^{10} ),
\end{eqnarray*}
which is an almost reverse lexicographical monomial ideal.
\end{example}

\bibliographystyle{acm}   


\end{document}